\documentclass[reqno]{amsart}
\usepackage[marginratio=1:1]{geometry}
\usepackage[hidelinks]{hyperref}
\usepackage{calc}
\newsavebox\CBox
\newcommand\hcancel[2][0.5pt]{%
  \ifmmode\sbox\CBox{$#2$}\else\sbox\CBox{#2}\fi%
  \makebox[0pt][l]{\usebox\CBox}%
  \rule[0.5\ht\CBox-#1/2]{\wd\CBox}{#1}}
\usepackage{blindtext}
\usepackage{color}
\usepackage{hyperref,url}
\usepackage{comment}
\usepackage{float}
\usepackage{hyperref}
\usepackage{enumerate}
\usepackage{enumitem}   
\usepackage{bbm}
\usepackage{cancel}
\usepackage{mathrsfs}
\usepackage{colonequals}
\usepackage{pdfpages}

\usepackage{amsmath,amsfonts,amsthm,bm}
\usepackage{mathrsfs}  
\usepackage{xcolor}
\usepackage{amsfonts}
\usepackage{amssymb}
\usepackage{amsmath}
\usepackage{mathtools}
\usepackage{mathrsfs}  

\usepackage[normalem]{ulem}

\numberwithin{equation}{section}
\usepackage[toc,page]{appendix}
\usepackage{tikz-cd}
\theoremstyle{definition}

\newtheorem{definicao}{Definition}[section]

\theoremstyle{plain}

\newtheorem{teorema}[definicao]{Theorem}
\newtheorem{suposicao}{Hypothesis}

\newtheorem{lema}[definicao]{Lemma}
\newtheorem{proposicao}[definicao]{Proposition}
\newtheorem{corolario}[definicao]{Corollary}

\newenvironment{TheoremproofA}[1]{\par\noindent{\emph{Proof of Theorem \ref{thm1}}} \space#1}{\leavevmode\unskip\penalty9999 \hbox{}\nobreak\hfill\quad\hbox{$\qed$}}
\newenvironment{TheoremproofB}[1]{\par\noindent{\emph{Proof of Corollary \ref{cor3}.}}  }{\leavevmode\unskip\penalty9999 \newline\hbox{}\nobreak\hfill \quad\hbox{$\qed$}}
\newenvironment{TheoremproofC}[1]{\par\noindent{\textit{Proof of Theorem \ref{thm2}}} \space#1}{\leavevmode\unskip\penalty9999 \hbox{}\nobreak\hfill\quad\hbox{$\qed$}}
\usepackage{scalerel}

\definecolor{roxo}{rgb}{0.44, 0.16, 0.39}
\definecolor{ao(english)}{rgb}{0.0, 0.5, 0.0}
\definecolor{dmagenta}{RGB}{139, 0, 139}
\definecolor{dgreen}{RGB}{0,90,0}
\definecolor{navy}{RGB}{0,0,128}

\usepackage{stackengine}

\definecolor{iblue}{RGB}{0, 35, 194}
\title[ANOVA]{Stability of quasi-stationary measures in high-dimensional products of mixing Markov chains}

\author[Matteo Tanzi]{Matteo Tanzi$^{2}$}
\email{\href{mailto:matteo.tanzi@kcl.ac.uk}{matteo.tanzi@kcl.ac.uk}}

\author[Giuseppe Tenaglia]{Giuseppe Tenaglia$^{1}$}
\email{\href{mailto:giuseppe.tenaglia20@imperial.ac.uk}{giuseppe.tenaglia20@imperial.ac.uk}}

\address{$^{2}$Department of Mathematics, King's College London, Strand, London WC2R 2LS, UK}

\address{$^{1}$Department of Mathematics, Imperial College London, London SW7 2AZ, UK}

\allowdisplaybreaks
\begin{document}

\subjclass[2020]{37C30, 37A30, 60J05}

\keywords{quasi-stationary measures,  high-dimensional dynamical systems, transfer operators, ANOVA/Hoeffding decomposition, spectral stability, Markov chains.}

\begin{abstract}
We study high-dimensional conditioned dynamics obtained from an arbitrary number of independent copies of a mixing Markov chain. The dynamics is conditioned to avoid a family of holes whose stationary measure vanishes as the dimension grows, and we ask whether the resulting process admits a quasi-stationary measure close to the stationary product measure. Our problem is motivated by conditioning on holes with complicated geometry in high dimension, such as sets arising naturally from large deviations of suitable observables.

Our main tool is the ANOVA decomposition, which separates functions according to their dependence on different subsets of coordinates. This allows us to exploit the product structure of the dynamics and obtain contraction estimates that are uniform in the dimension. Combined with a Keller–Liverani perturbation argument, these estimates yield the existence of quasi-stationary densities converging to the stationary density as the size of the hole vanishes. Under stronger one-step mixing and regularization assumptions, we obtain sharper convergence in a Sobolev norm, with a square-root dependence on the measure of the hole. The improvement relies on the increasingly strong contraction of higher-order ANOVA components, thereby overcoming the usual loss of control with dimension. Finally, we apply our results to additive-noise Markov chains on the circle with smooth, uniformly positive transition densities
\end{abstract}
\maketitle

\section{Introduction}
\label{intro}

A fundamental problem in dynamical systems is understanding the behavior of complex systems. These are high-dimensional systems arising from the interaction of an arbitrary number of particles, each with its own individual dynamics. An important phenomenon observed in high-dimensional systems is metastability, namely the coexistence of a unique asymptotic equilibrium with long-lived transient regimes. In many cases, the timescales over which such phenomena persist grow with the dimension of the system, so that, when the number of particles is very large, they can become essentially indistinguishable from genuine asymptotic behavior \cite{YasodharanSundaresan2023,DawsonSidAliZhao2023,LocherbachMonmarche2022}. In addition, if the underlying system is well approximated by some limiting dynamical system, when the number of particles tends to infinity, such transient phenomena can be experienced as equilibria by the limiting system \cite{CagliotiRousset2007,YamaguchiEtAl2004,LocherbachMonmarche2022}.

In a recent work \cite{JournelLeBris2025}, the authors study the continuous-time Curie--Weiss model, in a certain regime  where the $N$-spin system has a unique invariant measure, but the corresponding infinite particle limiting dynamics, obtained as $N \to \infty$, has two stable equilibria. In this work, they associate these equilibria with certain probability measures for the $N$-spin system, called quasi-stationary measures, which describe the statistics of the transient dynamics conditioned  not to enter a certain prescribed region, commonly referred to as a hole. The mass this measure retains outside the hole is called rate of survival: the closer  is to one, the longer one expects the transient phenomena to last. 
In \cite{JournelLeBris2025}, the authors prove that, given an equilibrium measure  for the mean-field system, the $N$-spin system conditioned not to enter, roughly, the complement of a suitable metastable region associated with this equilibrium admits a quasi-stationary measure converging  to the given equilibrium when $N \to \infty$, whilst its rate of survival converges to $1.$

This naturally raises the question of whether the result of \cite{JournelLeBris2025} extends to more general mean-field systems. We are particularly interested in answering this question for discrete-time mean-field coupled dynamical systems, whose underlying one-particle dynamics is either expanding or a random system driven by additive diffusive noise. Indeed, for this class, recent work has  begun to investigate the connection between relevant invariant or metastable states of the finite-dimensional dynamics and equilibria of its infinite particle limit \cite{CastorriniGalatoloTanzi2026,Tanzi2022}, building on an 
already rich literature regarding the latter \cite{SelleyBalint2016,SelleyTanzi2021,Galatolo2022,Tanzi2023}.

Establishing the correspondence in \cite{JournelLeBris2025} for our target class presents non-trivial challenges. The key issue is that one needs to control simultaneously the interaction between the particles and the conditioning induced by the hole. 
Beyond the uniformly expanding setting
treated in \cite{Tanzi2022}, however, dimension-uniform control of the
finite-dimensional interaction remains largely open, especially for
discrete-time systems with noise. Furthermore, despite the presence of a rich literature on conditioned dynamics \cite{PianigianiYorke1979,LiveraniMaumeDeschamps2003,CastroLambOliconMendezRasmussen2024,ChampagnatVillemonais2023}, substantially less is known about conditioned finite-dimensional dynamics in high-dimensional mean-field systems.

In this work, we study a problem of intermediate difficulty. Indeed, we disregard the mean-field interaction term and study the conditioned dynamics of an arbitrary number of independent copies of a mixing one-dimensional Markov chain on the circle $\mathbb{T}$. 
This allows us to isolate the technical difficulties arising from the
conditioning, while retaining a close connection with the motivating
mean-field model: although the latter is not a product system, for each particle, law-of-large-numbers-type phenomena make the interaction term progressively less sensitive to the specific configuration of the other particles as the dimension grows.

More precisely, for any $N\in \mathbb{N}$, we consider $N$ independent copies of a Markov chain in $\mathbb{T}$ satisfying one of the following mixing assumptions:
\begin{itemize}
\item  Hypothesis \ref{A}, in which we assume the chain admits a stationary measure $\mu$, such that its transfer operator $\mathcal{P}_{\mu}$ eventually contracts functions in $L^2(\mu)$ with zero $\mu$-integral;
\item  Hypothesis \ref{B}, in which we require a more quantitative notion of mixing: the existence of a stationary measure $\mu$ for which the transfer operator $\mathcal{P}_{\mu}$ contracts, in one step, all functions in $L^2(\mu)$ with zero $\mu$-integral.
\end{itemize}
Furthermore, in Hypothesis \ref{B}, we assume the regularizing property that the transfer operator is bounded from $L^2(\mu)$ to the Sobolev space $W^{1,2}(\mu)$ defined in \eqref{sobolevspace} (see \eqref{regularization}). This means essentially that there exists $C>0$ such that
\begin{equation}\label{aiutofuoco}
||\mathcal{P}_{\mu} ||_{L^2(\mu)\to W^{1,2}(\mu)} \le C.
\end{equation}

Such a property is naturally satisfied in the presence of sufficiently smooth noise \cite{Galatolo2022}.

Given a one-dimensional mixing chain $\{X_n\}_{n \ge 0}$ satisfying either Hypothesis $\ref{A}$ or Hypothesis \ref{B} with stationary measure  $\mu$, we denote by $\mu_N$ its $N$-fold product measure  and consider a family of holes $\{H_N\}_{N\ge 0}$ satisfying
\begin{equation}\label{cooldecay}
\mu_N(H_N) \to 0,\qquad N \to \infty.
\end{equation}
We ask whether the dynamics of the $N$ independent copies of $\{X_n\}_{n \ge 0}$, conditioned not to enter $H_N$, admits a quasi-stationary measure close  to $\mu_N$,  whose rate of survival converges to one.

Under the mixing assumption (Hypothesis \ref{A})  and the vanishing condition \eqref{cooldecay},  Theorem \ref{thm1}  establishes that, for all $p \in (1,\infty),\varepsilon \in \left(0,\frac{1}{p}\right) $, if $N$ is large enough, then the conditioned dynamics admits a  quasi-stationary measure with survival rate $\lambda_N$, whose density $\psi_N \in L^p(\mu_N)$  satisfies
\begin{equation}\label{sol1}
\|\psi_N-1 \|_{L^p(\mu_N)} \lesssim \mu_N(H_N)^{\frac{1}{p}-\varepsilon},\qquad |1-\lambda_N|\lesssim    \mu_N(H_N)^{1-\varepsilon} 
\end{equation}
Furthermore, under the stronger one-step  mixing and smoothing  assumption in Hypothesis \ref{B},  Theorem \ref{thm2} yields the stronger estimate
\begin{equation}\label{sol2}
\| \psi_N -1\|_{W^{1,2}(\mu_N)}\lesssim \sqrt{\mu_N(H_N)}, \qquad |1-\lambda_N| \lesssim \mu_N(H_N)
\end{equation}
where $W^{1,2}(\mu_N)$ is the Sobolev space defined in \eqref{Ndimsobolev}.

As an application of Theorem \ref{thm1} and \ref{thm2}, we consider a one-dimensional chain of the form
\begin{align*}
x_{n+1} = f(x_n) + \omega_n \mod 1, 
\end{align*}
where $\{\omega_n \}_{n}$ are i.i.d. random variables drawn from $[0,1]$ with a $C^1$ density $h$ equivalent to Lebesgue (see Hypothesis \ref{C} for a more precise definition). In Corollary \ref{cor3}, we prove, via a coupling argument, that all the Markov chains in this class satisfy Hypothesis \ref{B}, so that both Theorems 2.1 and 2.2 apply in this setting.

In order to obtain such results, we follow the operator-theoretic approach. Let  $\mathcal{P}_{\mu,N}$ be the transfer operator associated with the $N$-dimensional chain, and let $H_N$ be a hole. Consider the conditioned operator
\begin{equation}\label{op1}
\mathcal{P}_{\mu,H_N}(\phi):=\mathcal{P}_{\mu,N}(1_{\mathbb{T}^N\setminus H_N}\phi).
\end{equation}
Then, any normalized non-negative eigenfunction $\psi_N$ with  eigenvalue $\lambda_N>0$ of this operator is a quasi-stationary density for the high-dimensional process, conditioned not to enter $H_N$, and  $\psi_N d\mu_N$ is the quasi-stationary measure with rate of survival $\lambda_N$.
Equivalently, quasi-stationary densities correspond to non-negative fixed points of the normalized conditioned operator
\begin{equation}\label{op2}
\bar {\mathcal{P}}_{\mu,H_N}(\phi):=\mathcal{P}_{\mu,N}\left(\frac{1_{\mathbb{T}^N\setminus H_N}\phi}{|| 1_{\mathbb{T}^N\setminus H_N}\phi||_{L^1(\mu_N)}}\right)
\end{equation}
Our strategy is to compare either \eqref{op1} or \eqref{op2} with the unconditioned operator $P_{\mu,N}$, exploiting the vanishing condition \eqref{cooldecay} through a suitable perturbative argument.

To proceed with it, we need to resolve two issues. The first is to derive the properties of $\mathcal{P}_{\mu,N}$, which in our case is the transfer operator associated with $N$ independent copies of a mixing chain, using only the information on the one-dimensional chain and its operator $\mathcal{P}_{\mu}$ either in Hypothesis \ref{A} or Hypothesis \ref{B}.
The second is to find a suitable perturbative approach to compare $\mathcal{P}_{\mu,N}$ with either $\mathcal{P}_{\mu,H_N}$ or $\bar{\mathcal{P}}_{\mu,H_N}$ (see \eqref{op1} and \eqref{op2}), when $N$ is large enough. The classical  low-dimensional small holes perturbation arguments \cite{KellerLiverani1999} are not suitable to this task. Indeed, they rely on the one-dimensional embedding $BV \to L^{\infty}$, whose $N$-dimensional analogue  $BV \to L^{\frac{N}{N-1}}$
allows us to mimic the methods in \cite{KellerLiverani1999} either when the measure of the hole decays sufficiently fast, or when, as in \cite{BahsounSelley2022,BahsounPhalempin2026}, the hole retains a  particular local structure.

To solve these issues,  we introduce, in the analysis of high dimensional systems with holes, an orthogonal decomposition of $L^2(\mu_N)$ called ANOVA. This is a standard tool in applied mathematics and statistics \cite{Hoeffding1948,EfronStein1981,Sobol1993}, which has also been employed to study operators acting on high-dimensional probability spaces \cite{Mossel2010GaussianBounds}. 
The key property of the ANOVA decomposition used in this work is that
it decomposes  $L^2(\mu_N)$ in orthogonal  subspaces, which are preserved and eventually contracted by $\mathcal{P}_{\mu,N}$, when the underlying one-dimensional chain satisfies Hypothesis \ref{A} or \ref{B}.

To prove Theorem \ref{thm1}, we first use the ANOVA decomposition to obtain dimension-free contraction estimates for $\mathcal{P}_{\mu,N}$ in the subspace of $L^2(\mu_N)$ of $\mu_N$-mean zero functions. An interpolation inequality, along with a Keller-Liverani argument using the weak-strong pair $L^p(\mu_N)/L^q(\mu_N)$ with $p<q$ , gives a spectral gap for the conditioned operator in $\mathcal{P}_{\mu,H_N}$ (see \eqref{op1}), along with the precise estimates on the spectral radius and the spectral projections needed to prove \eqref{sol1}.

To prove Theorem \ref{thm2}, we use  the ANOVA decomposition twice. First we use it to show that the normalized conditioned operator $\bar{\mathcal{P}}_{\mu,H_N}$ (see \eqref{op2}) has a unique fixed probability density $\psi_N$ in a neighborhood of $1$ in $L^2(\mu_N)$ of size $\sim \sqrt{\mu_N(H_N)}$, then we use it again, along with the one-dimensional smoothing estimate \eqref{aiutofuoco}, to establish an analogous higher-dimensional smoothing estimate for $\mathcal{P}_{\mu,N}$, with a dimension-free constant. This estimate, along with the properties of  $\psi_N$, give  the result.

The last application of ANOVA in  Theorem \ref{thm2} is the most technical result of the paper.  It crucially uses the fact that the  ANOVA decomposes $L^2(\mu_N)$ into subspaces uniquely identified by the subset $u\subset \{1,\dots,N\}$, in such a way that functions in  the component associated with $u$ depend only on the coordinates indexed by $u$. Under the one-step mixing assumption for the underlying one-dimensional chain in Hypothesis \ref{B}, the operator $\mathcal{P}_{\mu,N}$ preserves and  contracts each component with a  contraction factor that decays exponentially with the number $|u|$ of active coordinates (see Lemma \ref{preserbationANOVA}). In particular, the rate of contraction becomes exponentially stronger with the number of active coordinates, and this allows us to obtain a dimension-free version of \eqref{aiutofuoco}, thus defeating the curse of dimensionality.

The paper is organized as follows. In Section \ref{Section2} we state our main results. In Section \ref{Section3}  we introduce the ANOVA decomposition and use it to establish dimension-free contraction estimates for the transfer operator $\mathcal{P}_{\mu,N}$ associated with the $N$ independent copies of the underlying one-dimensional chain. Sections \ref{Section4} and~\ref{Section5}  are devoted to the proofs of Theorems~\ref{thm1} and~\ref{thm2}, respectively. In Section~\ref{Section6}, we prove Corollary~\ref{cor3}, showing that the class of additive-noise chains considered in Hypothesis~\ref{C} satisfies the assumptions of our abstract results. Finally, Appendix \ref{appenix} contains the quantitative Keller--Liverani perturbation result used in the proof of Theorem~\ref{thm1}.

\section{Hypotheses and main results}\label{Section2}
Given a Markov chain $\{X_n\}_{n\geq 0}$ on an underlying probability space $(\Omega,\Sigma,\mathbb{P})$, taking values in $\mathbb{T}$, let $\{p(x,\cdot)\}_{x\in\mathbb{T}}$ denote its transition kernel, which is the family of probability measures on $\mathbb{T}$ satisfying
\begin{align*}
\mathbb{P}(X_{n+1} \in A|X_n = x) = p(x,A),\qquad \forall x \in \mathbb{T}, n \in \mathbb{N}.
\end{align*}
The transition kernel induces, on the space of finite measures in $\mathbb{T}$, the following one-step evolution
\begin{equation}\label{timeev}
\bar p(\nu)(A)
:=
\int_{\mathbb{T}}p(x,A)d\nu(x),\qquad \forall\, \text{measurable}\,\,A \subset \mathbb{T}.
\end{equation} 

Any stationary measure for $\{X_n\}_{n \ge 0}$ is a fixed point of $\bar p$. Note that,  if $\mu$ is a stationary measure, and $A \subset \mathbb{T}$ has $\mu$ measure $0$ then  by \eqref{timeev}, for $\mu$-a.e. $x \in \mathbb{T}$   $p(x,A)=0$. This, in particular, implies that the space of  measures absolutely continuous with respect to $\mu$ is left invariant by $\bar p$. When restricted to this space, the evolution of a measure $\nu$ via  $\bar p$ is equivalent to evolution of its density $\phi:= \frac{d\nu}{d\mu}$ via the transfer operator 
\begin{align*}
\mathcal P_{\mu}&\colon L^1(\mu) \mapsto L^1(\mu),\\
\mathcal P_{\mu}(\phi)
&:=
\frac{d\bar p\bigl((\phi d\mu)\bigr)}{d\mu}.
\end{align*}

Our first assumption is that $\{X_n\}_{n \ge 0}$ is a mixing one-dimensional Markov chain.

\begin{suposicao}\label{A}
The one-dimensional chain $\{X_n\}_{n\ge 0}$ admits a stationary  measure $\mu$ for which the associated transfer operator $\mathcal{P}_{\mu}$ satisfies, for some $C>0,\gamma \in (0,1)$ and for every $k\geq 0$
\begin{equation}\label{geometricergodicity}
\left|\mathcal P_\mu^k \phi \right|_{L^2(\mu)}
\leq
C\gamma^k\left|\phi \right|_{L^2(\mu)}, \qquad \forall \phi \in L^2(\mu) \colon \int \phi d\mu =0
\end{equation}
\end{suposicao}

Given a Markov chain $\{X_n\}_{n \ge 0}$ satisfying Hypothesis \ref{A}, with stationary measure $\mu$ and transfer operator $\mathcal{P}_{\mu}$,
for $N \in \mathbb{N}$, let $\mu_N$ denote the product  measure $\mu_N := \underbrace{\mu \otimes \cdots \otimes \mu}_{N\text{ times}}$ and let $\mathcal{P}_{\mu,N}$ be the $N$-fold tensor product operator, which is the unique bounded linear operator satisfying
\begin{equation}\label{uncoupledop}
\mathcal{P}_{\mu,N}\left(\otimes_{i=1}^N \phi_i\right) = \otimes_{i=1}^N\mathcal{P}_{\mu}(\phi_i),\qquad \forall \phi_1,\dots, \phi_N \in L^1(\mu).
\end{equation}
Note that $\mathcal{P}_{\mu,N}$ is the transfer operator associated with  the $N$-dimensional Markov chain  $\{X^N_n\}_{n \ge 0}$, formed from \(N\)-independent copies of the one-dimensional chain $\{X_n\}_{n \ge 0}$.  This chain satisfies
\begin{equation}\label{high-dim}
\mathbb{P}\left\{X_{n+1}^N \in A_1\times \dots \times A_N| X^N_n = (x_1,\dots,x_N)\right\} = \prod_{j=1}^N \mathbb{P}\left\{X_{n+1} \in A_j|X_n = x_j\right\}.
\end{equation}

A key property that will be used later, is that if $\mu$ is a stationary measure for $\{X_n\}_{n \ge 0}$, then, due to \cite[Theorem 13.2]{EisnerEtAl2015} for all $N \in \mathbb{N}$, the $N$-fold operator $\mathcal P_{\mu,N}$ is contractive in $L^p(\mu_N)$ for any $p \ge 1$:
\begin{equation}\label{usedlater}
||\mathcal{P}_{\mu,N} ||_{L^p(\mu_N)\to L^p(\mu_N)} \le 1.
\end{equation}

Given a measurable set $H_N\subset \mathbb{T}^N$, let us define the conditioned operator as
\begin{equation}\label{conditioned}
\mathcal{P}_{\mu,H_N}(\phi)(x) = \mathcal{P}_{\mu,N}(\phi 1_{\mathbb{T}^N\setminus H_N})(x).
\end{equation}
This operator encodes the dynamics of the process obtained  by conditioning the high-dimensional chain $\{X^N_n\}_{n \ge 0}$ in \eqref{high-dim}  on never entering  $H_N$. For this reason,  we refer to $H_N$ as a hole. For this conditioned process, we say that a probability density $\psi_N$ is a \emph{quasi-stationary density}, if $\psi_N$ is a non-negative eigenfunction associated with an eigenvalue $\lambda_N>0$ of $\mathcal{P}_{\mu,H_N}$. Then, $\psi_N d\mu_N$ is  a quasi-stationary measure. These measures generalize the notion of stationary measures for conditioned processes.

Our first result shows that, for every fixed $p\in(1,\infty)$, given an underlying chain $\{X_n\}_{n\ge0}$ satisfying Hypothesis \ref{A}, and given a family of holes
$\{H_N\}_{N\geq 0}$ satisfying
\begin{equation}\label{holesdecay}
\mu_N(H_N)\to 0,
\qquad N\to\infty,
\end{equation}
the spectra of $\mathcal P_{\mu,N}$ and $\mathcal P_{\mu,H_N}$, considered as operators on
$L^p(\mu_N)$, are eventually close. In particular, this implies that, for $N$ large enough, the process $\{X_n^N\}_{n\ge 0}$ in \eqref{high-dim}, conditioned not to enter $H_N$, has a quasi-stationary density $\psi_N$ close to $1$ in $L^p(\mu_N)$.

\begin{teorema}\label{thm1}
Let $\{X_n\}_{n \ge 0}$ be a Markov chain satisfying Hypothesis \ref{A}, with stationary measure $\mu$ and transfer operator $\mathcal P_\mu$. Let $\mathcal{P}_{\mu,N}$ be the $N$-fold operator of $\mathcal{P}_{\mu}$ as in \eqref{uncoupledop} and, given a sequence of holes $\{H_N\}_{N\ge0}$ satisfying \eqref{holesdecay}, let $\mathcal{P}_{\mu,H_N}$ be as in \eqref{conditioned}.
 Then, for all $p \in (1,\infty)$, $\varepsilon \in (0,\frac{1}{p})$, there exists $N_{p,\varepsilon} \in \mathbb{N}$, constants $C_{p,\varepsilon}, \bar C_{p,\varepsilon}>0$ and $\theta_{p} \in (0,1)$ such that, for all $N \ge N_{p,\varepsilon}$, 
 \begin{equation}\label{spectrum}
 \sigma_{L^p(\mu_N)}(\mathcal{P}_{\mu,H_N}) \subset B_{\theta_p}(0)\cup B_{\bar C_{p,\varepsilon}\mu_N(H_N)^{1-\varepsilon}}(1)
 \end{equation}
and there is a unique simple eigenvalue $\lambda_{N,p} \in B_{\bar C_{p,\varepsilon}\mu_N(H_N)^{1-\varepsilon}}(1)$, whose associated non-negative eigenfunction $\psi_{N,p}$, normalized so that
$\int \psi_{N,p}d\mu_N =1,$
satisfies
\begin{equation}\label{qsm}
||\psi_{N,p}-1||_{L^p} \le C_{p,\varepsilon} \mu_N(H_N)^{\frac{1}{p}-\varepsilon}, 
\end{equation}
\end{teorema}
The above result can be strengthened under stronger mixing assumptions for the underlying  one-dimensional chain, along with a smoothing estimate for its associated transfer operator.

Let us consider the  one-dimensional Sobolev space 
\begin{equation}\label{sobolevspace}
W^{1,2}(\mu):=\left\{\phi \in L^2(\mu): \phi' \in L^2(\mu)\right\},
\end{equation}
where $\phi'$ denotes the weak derivative of $\phi$, endowed with the norm
\begin{align*}
||\phi||_{W^{1,2}(\mu)}:= |\phi|_{L^2(\mu)}+|\phi'|_{L^2(\mu)}
\end{align*}

\begin{suposicao}\label{B}
The one-dimensional chain $\{X_n\}_{n\ge 0}$ admits a stationary  measure $\mu$, equivalent to Lebesgue, for which the following holds:  there exists $\bar \gamma\in(0,1)$ such that
\begin{equation}\label{one_step}
\left|\mathcal P_\mu (\phi)\right|_{L^2(\mu)}
\leq
\bar \gamma  \left|\phi \right|_{L^2(\mu)},\qquad \forall \phi \in L^2(\mu)\colon \int_{\mathbb{T}}\phi d\mu=0.
\end{equation}
Furthermore, the operator 
$\mathcal{P}_{\mu} \colon L^2(\mu) \to W^{1,2}(\mu)$ is bounded, which means that  for all $\phi \in L^2(\mu)$, $\mathcal{P}_{\mu}(\phi)$ has a weak derivative $\mathcal{P}_{\mu}(\phi)'$ satisfying, 
\begin{equation}\label{regularization}
||\mathcal{P}_{\mu}(\phi)' ||_{L^2(\mu)}\le C_r ||\phi||_{L^2(\mu)},
\end{equation}
for some $C_r>0$.
\end{suposicao}
Note that Hypothesis \ref{B} implies Hypothesis \ref{A}.
Consider, for all $N\in \mathbb{N}$, the $N$-dimensional Sobolev space 
\begin{equation}\label{Ndimsobolev}
W^{1,2}(\mu_N):=\left\{\phi \in L^2(\mu_N): |\nabla \phi| \in L^2(\mu_N)\right\}    
\end{equation}
endowed with the norm
\begin{align*}
||\phi||_{W^{1,2}(\mu_N)}:= |\phi|_{L^2(\mu_N)}+|\nabla \phi|_{L^2(\mu_N)}
\end{align*}
Under  Hypothesis \ref{B}, we obtain the following strengthening of Theorem \ref{thm1}
\begin{teorema}\label{thm2}
Let $\{X_n\}_{n \ge 0}$ be a Markov chain satisfying Hypothesis \ref{B}, with stationary measure $\mu$ and transfer operator $\mathcal{P}_{\mu}$. Let $\mathcal{P}_{\mu,N}$ be the $N$-fold operator of $\mathcal{P}_{\mu}$ as in \eqref{uncoupledop} and, given a sequence of holes $\{H_N\}_{N\ge0}$ satisfying \eqref{holesdecay}, let $\mathcal{P}_{\mu,H_N}$ be as in \eqref{conditioned}.
Then, there exists $N_0 \in \mathbb{N}$ such that, for all $N \ge N_0$, the process $\{X^N_n\}_{n \ge 0}$ satisfying \eqref{high-dim}, conditioned on never entering  $H_N$, has a quasi-stationary density $\psi_N$ satisfying, for some $D_0>0$ independent of $N$
\begin{equation}\label{sobolev}
||\psi_N-1||_{W^{1,2}(\mu_N)} \le  D_0\sqrt{\mu_N(H_N)}.
\end{equation}
Furthermore, its associated eigenvalue $\lambda_N$ satisfies
\begin{equation}\label{rateofsurv}
|\lambda_N-1| \le  (1+D_0) \mu_N(H_N).
\end{equation}
\end{teorema}

\subsection{Application to Doeblin chains}
In this section, we finally specialize the abstract results to a class of additive-noise Markov chains satisfying a Doeblin-type condition. Essentially, we consider Markov chains on the unit circle, with a $C^1$ transition density uniformly bounded away from $0$.  
\begin{suposicao}\label{C}
The process $\{Y_n\}_{n\ge 0}$ has the form
\begin{equation}\label{doeblin}
Y_{n+1} = f(Y_n)+\omega_n \mod 1,
\end{equation}
where $f \colon \mathbb{T} \to \mathbb{R}$ is measurable, and
the random variables $\{\omega_n\}_{n \ge 0}$  are i.i.d. and drawn from a density $h \in  C^1(\mathbb{T})$ satisfying
\begin{equation}\label{estimates p}
0<c_1\le h(x)\qquad \forall x\in [0,1],
\end{equation}
and 
\begin{equation}\label{diffest}
|h'(x)|\le c_2, \qquad \forall x\in [0,1],
\end{equation}
for some $c_1,c_2 >0$.
\end{suposicao}

\begin{corolario}\label{cor3}

Let $\{Y_n\}_{n\ge 0}$ be a Markov chain satisfying Hypothesis \ref{C}. Then it satisfies Hypothesis \ref{B}.
\end{corolario}

As a result, Theorem \ref{thm1} and Theorem \ref{thm2} can be applied when the underlying one-dimensional chain satisfies Hypothesis  \ref{C}.

\section{ANOVA}\label{Section3}
The aim of this section is to leverage an orthogonal decomposition of $L^2(\mu_N)$, called ANOVA, to obtain dimension-free contraction rates for the $N$-fold product $\mathcal{P}_{\mu,N}$ (see \eqref{uncoupledop}) of a transfer operator satisfying \eqref{geometricergodicity}. More precisely, we first establish that each component of ANOVA is indexed by a specific set $u\subset \{1,\dots,N\}$, and then show that the $N$-fold product of an operator contracting the $L^2$ norm of $\mu$-mean-zero functions, with rate $\kappa$, contracts each ANOVA subspace with rate $\kappa^{|u|}$. 

Given $u \subset \{1,\dots, N\}$ and $x \in \mathbb{T}^N$, let $x_u = (x_j)_{j \in u}$ denote the projection of $x$ onto the coordinates indexed by $u$. If $u = \emptyset$, let
\begin{align*}
S_{\emptyset} = \text{span}(1),
\end{align*}
otherwise, define $S_u \subset L^2(\mu_N)$  as the subspace of functions,  depending only on the coordinates in $u$, and that are  $\mu$-mean zero on each of these coordinates:
\begin{equation}\label{esseu}
S_u:=\left\{\phi=\phi(x_u)\in L^2(\mu_N)\colon  \int \phi(x_u)\,d\mu(x_j)=0
\ \text{for $\mu_{|u|-1}$-a.e. }x_{u\setminus\{j\}},\ \forall j\in u\right\}.
\end{equation}
Take  $\phi \in L^2(\mu_N)$,  and  define the functions $(\phi_u)_{u \subseteq \{1,\dots,N\}}$ recursively as follows. For the empty set, define
\begin{equation}\label{mean-zero}
\phi_{\emptyset} := \int_{\mathbb{T}^N} \phi(x)\, d\mu_N(x).
\end{equation}
If $|u| \ge 1$
\begin{equation}\label{more-coord}
\phi_u(x_u)
=
\int_{\mathbb{T}^{ N-|u|}} \phi(x_u,x_{-u})\, d\mu_{N-|u|}(x_{-u})
-
\sum_{v \subsetneq u} \phi_v(x_v),
\end{equation}
where $x_{-u} = (x_j)_{j \notin u}$. The following proposition collects standard results. For Lebesgue reference measure, a proof is given in \cite[Appendix A]{owen2013montecarlo}, and the argument extends verbatim to the present product probability space
\begin{proposicao}\label{ANOVAprop}
The subspaces $S_u$ form an orthogonal decomposition of $L^2(\mu_N)$: 
\begin{align*}
L^2(\mu_N) = \bigoplus_{u \subset \{1,\dots,N\}} S_u.
\end{align*}
Furthermore, for all $u \subset \{1,\dots,N\}$, the function $\phi_u$ in \eqref{mean-zero}-\eqref{more-coord}  belongs to $S_u$ and  \begin{equation}\label{represntation}
\phi(x) = \sum_{u \subseteq \{1,\dots,N\}} \phi_u(x_u).
\end{equation}
In particular, the Parseval identity holds, namely
\begin{equation}\label{parseval}
||\phi||^2_{L^2(\mu_N)} = \sum_{u \subseteq \{1,\dots,N\}} ||\phi_u||^2_{L^2(\mu_N)}.
\end{equation}
\end{proposicao}

The following lemma shows that, if a transfer operator associated with a Markov chain with stationary measure $\mu$ contracts the $L^2(\mu)$ norm of $\mu$-mean-zero functions at rate $\kappa$, then its $N$-fold product contracts the ANOVA subspace $S_u$ at rate $\kappa^{|u|}$.
 
Recall that, given a linear operator $\mathcal{Q}_{\mu} \colon L^2(\mu) \to L^2(\mu)$, its $N$-fold product $\mathcal{Q}_{\mu,N}$ is the unique bounded linear operator satisfying $$\mathcal{Q}_{\mu,N}\left(\otimes_{i=1}^N \phi_i\right) = \otimes_{i=1}^N\mathcal{Q}_{\mu}(\phi_i),\qquad \forall \phi_1,\dots, \phi_N \in L^2(\mu).$$
\begin{lema}\label{preserbationANOVA}
Let 
$\mathcal{Q}_{\mu} \colon L^2(\mu) \to L^2(\mu)$ be an operator satisfying $\mathcal{Q}_{\mu}(1) = 1$ and
\begin{equation}\label{pres}
\int Q_{\mu}(\phi) d\mu = \int \phi d\mu \qquad \forall \phi \in L^2(\mu)
\end{equation}
\begin{equation}\label{contrhelp}
||\mathcal{Q}_{\mu}(\phi) ||_{L^2(\mu)} \le \kappa ||\phi||_{L^2(\mu)},\qquad \forall \phi \in L^2(\mu)\colon \int \phi d\mu =0,
\end{equation}
 for some $\kappa \in (0,1]$. Then, for all  $u \subset \{1,\dots,N\}$, the $N$-fold operator $\mathcal{Q}_{\mu,N}$ satisfies
\begin{equation}\label{preservance}
\mathcal{Q}_{\mu,N}(S_u) \subset S_u
\end{equation}
and
\begin{equation}\label{conraction}
||\mathcal{Q}_{\mu,N}(\phi)  ||_{L^2(\mu_N)} \le \kappa^{|u|}||\phi  ||_{L^2(\mu_N)},\qquad \forall \phi \in S_u
\end{equation}
\end{lema}

\begin{proof}
For each $j\in{1,\ldots,N}$, let
$$
\mathcal Q_{j}
:=
I^{\otimes(j-1)}
\otimes\mathcal Q_{\mu}
\otimes I^{\otimes(N-j)}
$$
denote the operator acting as $\mathcal Q_{\mu}$ on the $j$-th coordinate and as the identity on all the remaining coordinates. Then
\begin{equation}\label{tensordec}
\mathcal Q_{\mu,N}= \mathcal Q_{1}\circ\cdots\circ\mathcal Q_{N}.    
\end{equation}

Let  $\phi\in S_u$. We show
$$
\mathcal Q_{j}\phi\in S_u, \qquad \forall j \in \{1,\dots,N\}.
$$
Indeed, if $j \notin u$, then $\phi$ does not depend on the variable $x_j$. Hence
$$
\mathcal Q_{j}\phi=\phi \mathcal Q_{j}(1)= \phi.
$$

If $j\in u$, then, if $i\in u$ and  $i\neq j$, by Fubini's theorem 
$$
\int_{\mathbb{T}}
\mathcal Q_{j}\phi(x_u)d\mu(x_i)
= \mathcal Q_{j}
\left(
\int_{\mathbb{T}}
\phi(x_u)d\mu(x_i)
\right)= 0.
$$
If $i=j$, by \eqref{pres}
$$
\int_{\mathbb{T}}
\mathcal Q_{j}\phi(x_u)d\mu(x_j)
= \int_{\mathbb{T}}
\phi(x_u)d\mu(x_j)
=0.
$$
Therefore,
$$
\mathcal Q_{j}(S_u)\subseteq S_u
$$
for every $j= 1,\dots,N$. Because of \eqref{tensordec}, \eqref{preservance} holds.

We now prove the contraction estimate. Fix $j\in u$ and  write $x_u = (x_{u\setminus{j}},x_j)$. By the definition of $S_u$, $\mu_{|u|-1}$-a.e. fixed
$x_{u\setminus{j}}$, the function
\begin{align*}
\phi_{x_{u\setminus j}}(x_j):=\phi(x_{u\setminus{j}},x_j)
\end{align*}
has zero $\mu$-mean. Hence, by \eqref{contrhelp},
\begin{equation}\label{fubini_trick}
\begin{aligned}
\left|
\mathcal Q_{j}\phi
\right|_{L^2(\mu_N)}^2
&=
\int_{\mathbb T^{|u|-1}}
\left|
\mathcal Q_{\mu}
\left(
\phi_{x_{u\setminus j}}
\right)
\right|_{L^2(\mu)}^2
d\mu_{|u|-1}
\left(
x_{u\setminus{j}}
\right)
\\
&\leq
\kappa^2
\int_{\mathbb T^{|u|-1}}
\left|
\phi_{x_{u\setminus j}}
\right|_{L^2(\mu)}^2
d\mu_{|u|-1}
\left(
x_{u\setminus{j}}
\right)
\\
&=
\kappa^2
|\phi|_{L^2(\mu_N)}^2.
\end{aligned}
\end{equation}
Therefore,
$$
\left|
\mathcal Q_{j}\phi
\right|_{L^2(\mu_N)}
\leq
\kappa
|\phi|_{L^2(\mu_N)}.
$$

Since each $\mathcal Q_{j}$ preserves $S_u$, we may apply the above estimate successively to all the coordinates $j\in u$, which proves \eqref{conraction}.
\end{proof}
We conclude this section with the dimension-free contraction rate for the $N$-fold operator  associated with a one-dimensional mixing chain as in Hypothesis \ref{A} . 

\begin{proposicao}\label{onedimrate}
Let $\{X_n\}_{n \ge 0}$ be a Markov chain satisfying Hypothesis \ref{A}, with stationary measure $\mu$ and transfer operator $\mathcal{P}_{\mu}$, and let $\mathcal{P}_{\mu,N}$ denote the $N$-fold operator associated with $\mathcal{P}_{\mu}$. Then, for the same $C,\gamma$ as in \eqref{geometricergodicity}
\begin{equation}\label{dimfree}
||\mathcal{P}_{\mu,N}^n(\phi) ||_{L^2(\mu_N)} \le C \gamma^n || \phi||_{L^2(\mu_N)},\qquad \forall \phi \in L^2(\mu_N)\colon \int \phi d\mu_N=0.
\end{equation}
\end{proposicao}

\begin{proof}
Let $C>0,\gamma \in (0,1)$ as \eqref{geometricergodicity}. Due to \eqref{usedlater}, for all $n \ge 0$
\begin{align*}
||\mathcal{P}^n_{\mu}(\phi) ||_{L^2(\mu)} \le \min\{1,C\gamma^n\}|| \phi||_{L^2(\mu)},\qquad \forall \phi \in L^2(\mu)\colon \int \phi d\mu=0.
\end{align*}

For $n \in \mathbb N$, let $\kappa_n:=\min\{1,C\gamma^n\}$ and   $\mathcal{R}_n := \mathcal{P}_{\mu}^n$. Furthermore, let $R_{n,N}$ denote the $N$-fold product of $R_n$, i.e. $R_{n,N}=P_{\mu,N}^n$.  If $\phi \in L^2(\mu_N)$ has $\mu_N$-zero mean, due to \eqref{mean-zero} and \eqref{represntation},  the ANOVA decomposition reads
\begin{align*}
\phi = \sum_{\emptyset \neq u \subset \{1,\dots,N\}} \phi_u,
\end{align*}
with $\phi_u$ as in \eqref{more-coord}.
By linearity
\begin{align*}
\mathcal{R}_{n,N}(\phi) = \sum_{\emptyset \neq u \subset \{1,\dots,N\}} \mathcal{R}_{n,N}(\phi_u).    
\end{align*}
By the Parseval identity in \eqref{parseval}, and \eqref{conraction} with $\mathcal{Q}_{\mu}=\mathcal R_n$, we get 
\begin{align*}
||\mathcal R_{n,N}(\phi)||^2_{L^2(\mu_N)} &= \sum_{\emptyset \neq u \subset \{1,\dots,N\}} ||\mathcal R_{n,N}(\phi_u)||^2_{L^2(\mu_N)} \\
&\le 
\sum_{\emptyset \neq u \subset \{1,\dots,N\}} \left(\kappa_n\right)^{2|u|}||(\phi_u)||^2_{L^2(\mu_N)}\\
&\le \min\{C\gamma^n,1\}^2||\phi||^2_{L^2(\mu_N)},
\end{align*}
where in the last inequality we used again \eqref{parseval}.
\end{proof}

\section{Proof of Theorem \ref{thm1}}\label{Section4}
Let  $\{H_N\}_{N \ge 0}$ be a sequence of holes such that $\mu_N(H_N) \to 0$ as $N\to \infty$. To prove Theorem \ref{thm1}, we compare, for $N$ large enough, the spectra of $\mathcal{P}_{\mu,N}$ and $\mathcal{P}_{\mu,H_N}$ by applying the stochastic stability result, established in Proposition \ref{Kl}, which is a quantitative version of the classical Keller-Liverani framework \cite{KellerLiverani1999}.

Leveraging the dimension-free rate of contraction established in Proposition \ref{onedimrate}, the following proposition checks that the operators $\mathcal{P}_{\mu,N}$ and $\mathcal{P}_{\mu,H_N}$ satisfy the assumptions \eqref{continuousKL}-\eqref{lasotayorkeKL} requested to apply the stability result in Proposition \ref{Kl}.

\begin{proposicao}\label{checklive}
Let $\{X_n\}_{n \ge 0}$ be a one-dimensional Markov chain satisfying Hypothesis \ref{A}, with its associated stationary measure $\mu$ and transfer operator $\mathcal{P}_{\mu}$.  Let $\mathcal{P}_{\mu,N}$ be the $N$-fold operator associated with $\mathcal{P}_{\mu}$ as in \eqref{uncoupledop} and, given a sequence of holes $\{H_N\}_{N \ge 0}$ satisfying \eqref{holesdecay}, let $\mathcal{P}_{\mu,H_N}$ be the conditioned operator as in \eqref{conditioned}. For all $q \in (1,\infty)$, there exist $C_q>0,$ $\gamma_q \in(0,1)$, independent of $N$, such that, for all $n\ge 0$
\begin{equation}\label{unifcontrlp}
\left|\left|\mathcal{P}_{\mu,N}^n(\phi)-\int \phi d\mu_N \right|\right|_{L^q(\mu_N)} \le C_q \gamma_q^n||\phi||_{L^q(\mu_N)}, \qquad \forall \phi \in L^q(\mu_N).
\end{equation}
Furthermore, there exists $C_{1,q},C_{2,q}>0$, independent of $N$, such that  for all $p \in [1,q)$ and all $n\ge 0$
\begin{equation}\label{lasota}
\left|\left| \mathcal{P}_{\mu,H_N}^n(\phi)\right|\right|_{L^q(\mu_N)} \le  C_{1,q}\gamma_{q}^n||\phi ||_{L^q(\mu_N)}+C_{2,q}||\phi||_{L^p(\mu_N)},
\end{equation}
with $\gamma_q$ as in \eqref{unifcontrlp}
and
\begin{equation}\label{Kellerlivdiff}
\sup_{\left\{\phi \colon ||\phi||_{L^q(\mu_N)}=1 \right\}}||(\mathcal{P}_{\mu,N}-\mathcal{P}_{\mu,H_N})(\phi) ||_{L^p(\mu_N)} \le \mu_N(H_N)^{\frac{1}{p}-\frac{1}{q}}
\end{equation}
\end{proposicao}
\begin{proof}
First, note that \eqref{dimfree} implies \eqref{unifcontrlp} with $C_2= C$ and $\gamma_2 =\gamma$ and
\begin{align*}
\Pi_N(\phi):= \int \phi d\mu_N.
\end{align*}
By invariance of $\mu$, we have, for all $r \in [1,\infty]$
\begin{align*}
||\mathcal{P}_{\mu,N}^n -\Pi_N||_{L^r(\mu_N)} \le 2, \qquad \forall n\ge 0.
\end{align*}
The above equation, together with \eqref{dimfree} and  the Riesz--Thorin interpolation theorem
\cite{BerghLofstrom1976}, applied between $L^1$ and $L^2$ for $1<q\leq 2$,
and between $L^2$ and $L^\infty$ for $2\leq q<\infty$, yields
\eqref{unifcontrlp}, with
\[
C_q=2^{1-\alpha_q}C^{\alpha_q},
\qquad
\gamma_q= \gamma^{\alpha_q},
\qquad
\alpha_q
=2\min\left\{\frac{1}{q},1-\frac{1}{q}\right\}.
\]
To establish \eqref{lasota}, observe that, by \eqref{unifcontrlp}
\begin{align*}
||\mathcal{P}^n_{\mu,H_N}(\phi)||_{L^q(\mu_N)} &\le ||\mathcal{P}_{\mu,N}^n(|\phi|)||_{L^q(\mu_N)} \\
&\le |\phi|_{L^1(\mu_N)}+ \left|\left|\mathcal{P}_{\mu,N}^n(|\phi|)-\int |\phi| d\mu_N\right|\right|_{L^q(\mu_N)} \\
&\le |\phi|_{L^1(\mu_N)}+ C_q\gamma_q^n||\phi||_{L^q(\mu_N)}\\
&\le |\phi|_{L^p(\mu_N)}+ C_q\gamma_q^n||\phi||_{L^q(\mu_N)},
\end{align*}
for all $n \ge 0$, which implies \eqref{lasota} with $C_{1,q}=C_q,C_{2,q}=1$.

To conclude the proof of the Proposition, note that, using invariance of $\mu$, and Holder inequality, for $q>p$
\begin{align*}
||(\mathcal{P}_{\mu,N}-\mathcal{P}_{\mu,H_N})(\phi) ||_{L^p(\mu_N)}&= ||\mathcal{P}_{\mu,N}(1_{H_N}\phi)||_{L^p(\mu_N)}\\
&\le ||1_{H_N}\phi||_{L^p(\mu_N)}
\\
&\le \left(\mu_N(H_N)\right)^{\frac{1}{p}-\frac{1}{q}}||\phi||_{L^q(\mu_N)},
\end{align*}
which shows \eqref{Kellerlivdiff}.
\end{proof}

\begin{TheoremproofA}
Fix $p \in (1,\infty)$, $\varepsilon \in \left(0,\frac{1}{p}\right)$ and the sequence of holes $\{H_N\}_{N\ge 0}$ satisfying \eqref{holesdecay}. Fix  $q>p$ such that 
$\frac{1}{q}<\varepsilon$. We want to apply  the stability result Proposition \ref{Kl} with the choices  
$$\widetilde{\mathcal{B}} = L^p(\mu_N),\,\mathcal{B} = L^q(\mu_N),\, \mathcal{P}_0 =\mathcal{P}_{\mu,N},\, \mathcal{P}_{\delta}=\mathcal{P}_{\mu,H_N} ,$$
the projection operator
$$\Pi_0(\phi):= \int \phi d\mu_N,$$
$\gamma= \gamma_q$ as in \eqref{unifcontrlp},
$\gamma' =\gamma'_q \in (\gamma_q,1)$ close enough to $1$ and $\beta=\beta_q$ close enough to zero such that
\begin{equation}\label{uppabound}
\left(\frac{1}{p}-\frac{1}{q}\right)(\eta-\beta_q)> \frac{1}{p}-\varepsilon,
\end{equation}
where $\eta$ is defined in \eqref{etaKL}. This choice is possible as $\eta \to 1$, when $\gamma' \to 1$. Note that also, as stated in Proposition \ref{Kl}, $\beta_q$ must satisfy $0<\beta_q<\frac{\eta}{2}.$

In order to do that, we show that, with the  choices of $\widetilde{\mathcal{B}},\mathcal{B},\mathcal{P}_0,\mathcal{P}_{\delta}$ and $\Pi_0$ as above,  all the estimates \eqref{continuousKL}-\eqref{nonespansionKl} necessary to apply Proposition \ref{Kl} are satisfied, with constants independent of $N$. 
Indeed, \eqref{continuousKL} follows from Holder-inequality,
\eqref{mixingKL}, follows from \eqref{unifcontrlp} with $C=C_q,\gamma=\gamma_q$, \eqref{lasotayorkeKL} follows from \eqref{lasota} with $C_1,C_2 = C_{1,q},C_{2,q}$ and $\gamma_1=\gamma_{q}$, and \eqref{nonespansionKl} follows from \eqref{usedlater}. 
Furthermore, observe that by \eqref{Kellerlivdiff}, \eqref{perturbationKL} is satisfied with $$\delta=\delta_N:= \left(\mu_N(H_N)\right)^{\frac{1}{p}-\frac{1}{q}}$$.

Because of \eqref{holesdecay}, by Proposition \ref{Kl}, there exist $N_{1,p,\varepsilon}\in \mathbb{N},C_{p,\varepsilon}>0$, depending only on $p$, the constants in \eqref{continuousKL}-\eqref{lasotayorkeKL} and the value of $q=q(p,\varepsilon)$ chosen at the beginning of the proof, such that, if $N\ge N_{1,p,\varepsilon}$ 
\begin{equation}\label{spectrumLq}
\sigma_{L^q(\mu_N)}(\mathcal{P}_{\mu,H_N}) \subset B_{\gamma'_q}(0) \cup B_{\delta_N^{\beta_q}}(1),
\end{equation}
and there exists a unique simple eigenvalue $\lambda_{N,q}\in B_{\delta_N^{\beta_q}}(1)$, which is the spectral radius of  $\mathcal{P}_{\mu,H_N}$ in $L^q(\mu_N)$. Since $\mathcal{P}_{\mu,H_N}$ is a positive operator, by \cite[Lemma 4.8]{daners2016eventually} the  normalized eigenfunction $\psi_{N,q} \in L^q(\mu_N)$ is a non-negative function and satisfies
\begin{equation}\label{qsmhelp}
||\psi_{N,q}-1||_{L^p(\mu_N)} \le C_{p,\varepsilon}\left(\mu_N(H_N)\right)^{\frac{1}{p}-\varepsilon}.
\end{equation}
Furthermore, since
\begin{align*}
1-\lambda_{N,q} = \int_{H_N} \psi_{N,q}d\mu_N,
\end{align*}
then, by Holder-inequality and \eqref{qsmhelp}, there exists $\bar{C}_{p,\varepsilon}>0$ such that

\begin{equation}\label{roshelp}
|\lambda_{N,q}-1| \le \bar C_{p,\varepsilon}\left(\mu_N(H_N)\right)^{1-\varepsilon}.
\end{equation}
Note that, since $p<q$, then $\lambda_{N,q}$ is an eigenvalue also for $\mathcal{P}_{\mu,H_N}$ in $L^p(\mu_N)$,  and \eqref{qsm} follows from \eqref{qsmhelp}. 

Another application of Proposition \ref{Kl} with 
$$
\widetilde{\mathcal B} = L^r(\mu_N),\,\mathcal{B} = L^p(\mu_N),
$$
with some $1<r<p$, gives values $\gamma'_p>\gamma_p$, $c_p>0, N_{2,p} \in \mathbb N$ such that, if $N\ge N_{2,p}$, then 
\begin{equation}\label{Halfspectrum}
\sigma_{L^p(\mu_N)}(\mathcal{P}_{\mu,H_N}) \subset B_{\gamma'_p}(0) \cup B_{\left(\mu_N(H_N)\right)^{c_p\left(\frac{1}{r}-\frac{1}{p}\right)}}(1),
\end{equation}
and that there exists a unique simple  eigenvalue $\lambda_{N,p} \in B_{\left(\mu_N(H_N)\right)^{c_p\left(\frac{1}{r}-\frac{1}{p}\right)}}(1)$. By uniqueness, then $\lambda_{N,p}=\lambda_{N,q}$ and \eqref{spectrum} follows from \eqref{roshelp},\eqref{Halfspectrum}, whenever $N\ge N_{p,\varepsilon}:=\max\{N_{1,p,\varepsilon},N_{2,p}\}$
\end{TheoremproofA}

\section{Proof of Theorem \ref{thm2}}\label{Section5}
Let $\{X_n\}_{n \ge 0}$ be a one-dimensional Markov chain satisfying Hypothesis \ref{B}, with stationary measure $\mu$ and transfer operator $\mathcal{P}_{\mu}$. Due to the one-step contraction for $\mathcal{P}_{\mu}$ in \eqref{one_step}  and the tensorization estimate in  Proposition \ref{onedimrate}, for all
\(\phi \in L^2(\mu_N)\) satisfying \(\int \phi \,d\mu_N=0\), one has
\begin{equation}
\label{eq:tensorized-contraction}
\|\mathcal P_{\mu,N}(\phi)\|_{L^2(\mu_N)}
\leq
\bar \gamma\|\phi\|_{L^2(\mu_N)}.
\end{equation}
The above fact will be used several times in the proof of Theorem \ref{thm2}.

Given a sequence of holes $\{H_N\}_{N\ge 0}$ satisfying \eqref{holesdecay}, we consider the conditioned normalized operator 
\begin{equation}\label{condnormal}
\bar{\mathcal{P}}_{\mu,H_N}(\phi) = \frac{\mathcal{P}_{\mu,N}(\phi1_{\mathbb{T}^N\setminus H_N})(x)}{||\phi 1_{\mathbb{T}^N\setminus H_N}||_{L^1(\mu_N)}} = \frac{\mathcal{P}_{\mu,H_N}(\phi)(x)}{||\phi 1_{\mathbb{T}^N\setminus H_N}||_{L^1(\mu_N)}},
\end{equation}
where $\mathcal{P}_{\mu,N}$ is the $N$-fold operator as in \eqref{uncoupledop} and $\mathcal{P}_{\mu,H_N}$ is the conditioned operator  in \eqref{conditioned}. It is easy to see that non-negative fixed points of $\bar{\mathcal{P}}_{\mu,H_N}$ are non-negative eigenfunctions for the conditioned operator $\mathcal{P}_{\mu,H_N}$.
Given $\bar{\gamma}$ as in \eqref{one_step}, define, for all $N \in \mathbb{N}$ 
\begin{equation}\label{deltaenne}
\delta_{N}:= \frac{4(\bar \gamma)}{1-\bar{\gamma}} \sqrt{\mu_N(H_N)},
\end{equation}
and the following  neighborhood of $1$ in $L^2(\mu_N)$
\begin{equation}\label{bienne}
B_{N}
:=
\left\{
\phi \in L^2(\mu_N)
\,:\,
\phi\ge 0,\,
\int \phi\,d\mu_N=1,\,
\|\phi-1\|_{L^2(\mu_N)}\le \delta_{N}
\right\}.
\end{equation}
The next proposition shows that, for $N$ large enough, $\bar {\mathcal{P}}_{\mu,H_N}$ has  a unique  fixed point $\psi_N$ in $B_N$. 
\begin{proposicao}
\label{inv_ne}
Let $\{X_n\}_{n \ge 0}$ be a one-dimensional Markov chain satisfying Hypothesis \ref{B}, with stationary measure $\mu$ and transfer operator $\mathcal{P}_{\mu}.$ 
Given a sequence of holes $\{H_N\}_{N\ge 0}$ satisfying \eqref{holesdecay},  there exists $N_0$ such that for any $N \ge N_0$
\begin{equation}\label{invariance}
\bar{\mathcal{P}}_{\mu,H_N}(B_N) \subset B_N,
\end{equation}
where $\bar{\mathcal{P}}_{\mu,H_N},B_N$  are as in \eqref{condnormal} and \eqref{bienne}, and there exists $\tilde{\gamma} \in (0,1)$ such that, if $\phi_1,\phi_2 \in B_N$,
\begin{equation}\label{strongcontr}
|\bar{\mathcal{P}}_{\mu,H_N}(\phi_1)-\bar{\mathcal{P}}_{\mu,H_N}(\phi_2)|_{L^2(\mu_N)} \le \tilde \gamma |\phi_1-\phi_2|_{{L^2(\mu_N)}}
\end{equation}
In particular, $\bar{\mathcal{P}}_{\mu,H_N}$  has a unique fixed point $\psi_N$ in $B_N$.
\end{proposicao}
\begin{proof}
Let $\mathcal{P}_{\mu}$ the transfer operator associated with $\{X_n\}_{n \ge 0}$. Furthermore, let $\mathcal{P}_{\mu,N}$ denote the $N$-fold operator associated with $\mathcal{P}_{\mu},$ and  consider the normalized conditioned operator $\bar{\mathcal{P}}_{\mu,H_N}$ as defined in \eqref{condnormal}. 
We first show \eqref{invariance}. 
Let $\phi \in B_N$, where $B_N$ is as in \eqref{bienne}. Then, due to \eqref{eq:tensorized-contraction}
\begin{equation}\label{prelim}
\begin{aligned}
\|\bar{\mathcal{P}}_{\mu,H_N}(\phi)-1\|_{L^2(\mu_N)}
&=
\left\|
\mathcal{P}_{\mu,N}
\left(
\frac{\phi 1_{\mathbb T^N\setminus H_N}}
{\|\phi 1_{\mathbb T^N\setminus H_N}\|_{L^1(\mu_N)}}
-1
\right)
\right\|_{L^2(\mu_N)}
\\
&\le
\bar\gamma
\left\|
\frac{\phi 1_{\mathbb T^N\setminus H_N}}
{\|\phi 1_{\mathbb T^N\setminus H_N}\|_{L^1(\mu_N)}}
-1
\right\|_{L^2(\mu_N)}
\\
&\le
\bar\gamma
\frac{
\|\phi-1\|_{L^2(\mu_N)}
+\sqrt{\mu_N(H_N)}
+\|1_{H_N}\phi\|_{L^1(\mu_N)}
}{
1-\|1_{H_N}\phi\|_{L^1(\mu_N)}
}.
\end{aligned}
\end{equation}
Let $\delta_N$ as in \eqref{deltaenne}. Then, if $\phi \in B_N$, $||\phi-1||_{L^2(\mu_N)} \le \delta_N$ and by Cauchy-Schwarz
\begin{equation}\label{eq:mass-hole-bound}
||1_{H_N}\phi||_{L^1(\mu_N)}
\leq 
\sqrt{\mu_N(H_N)}
\|\phi\|_{L^2(\mu_N)}
\leq
(1+\delta_N)\sqrt{\mu_N(H_N)}.
\end{equation}

Consequently, by \eqref{eq:mass-hole-bound} and \eqref{prelim}
\begin{align*}
\|\bar{\mathcal P}_{\mu, H_N}(\phi)-1\|_{L^2(\mu_N)}
\leq \bar \gamma \left(
\frac{
\delta_N+(2+\delta_N)\sqrt{\mu_N(H_N)}
}{
1-(1+\delta_N)\sqrt{\mu_N(H_N)}
}\right),
\end{align*}
which implies \eqref{invariance} for $N$ large enough.

We now prove the contraction estimate in \eqref{strongcontr}. Let
\(\phi_1,\phi_2\in  B_N\). Then, by \eqref{eq:tensorized-contraction}
\begin{equation}\label{eq:contraction-reduction}
\|\bar{\mathcal P}_{\mu,H_N}(\phi_1)
-\bar{\mathcal P}_{\mu,H_N}(\phi_2)\|_{L^2(\mu_N)}
\leq
\bar \gamma
\left|\left| \frac{\phi_1 1_{\mathbb T^N\setminus H_N}}
{\|\phi_1 1_{\mathbb T^N\setminus H_N}\|_{L^1(\mu_N)}}-\frac{\phi_2 1_{\mathbb T^N\setminus H_N}}
{\|\phi_2 1_{\mathbb T^N\setminus H_N}\|_{L^1(\mu_N)}}  \right|\right|_{L^2(\mu_N)}.
\end{equation}
We write
\begin{equation}\label{formuladiff}
\begin{aligned}
\frac{\phi_1 1_{\mathbb T^N\setminus H_N}}
{\|\phi_1 1_{\mathbb T^N\setminus H_N}\|_{L^1(\mu_N)}}-\frac{\phi_2 1_{\mathbb T^N\setminus H_N}}
{\|\phi_21_{\mathbb T^N\setminus H_N}\|_{L^1(\mu_N)}} 
&=
\frac{1_{\mathbb T^N\setminus H_N}(\phi_1-\phi_2)}{1-\|1_{H_N}\phi_1\|_{L^1(\mu_N)}}
\\
&+
1_{\mathbb T^N\setminus H_N}\phi_2
\left(
\frac{\|1_{H_N}\phi_1\|_{L^1(\mu_N)}-\|1_{H_N}\phi_2\|_{L^1(\mu_N)}}{\left(1-|1_{H_N}\phi_1\|_{L^1(\mu_N)}\right) \left(1-|1_{H_N}\phi_2\|_{L^1(\mu_N)}\right)}
\right).
\end{aligned}
\end{equation}
Since $\phi_1,\phi_2 \in B_N$, we have 
\begin{align}
\left|\|1_{H_N}\phi_1\|_{L^1(\mu_N)}-\|1_{H_N}\phi_2\|_{L^1(\mu_N)}\right|
=
\left|
\int_{H_N}(\phi_1-\phi_2)\,d\mu_N
\right|
\le \sqrt{\mu_N(H_N)}||\phi_1-\phi_2 ||_{L^2(\mu_N)}
\label{eq:mass-difference}
\end{align}
By \eqref{eq:mass-hole-bound},  \eqref{eq:contraction-reduction}, \eqref{formuladiff} and \eqref{eq:mass-difference} we have
\begin{align*}
\left|\left| \frac{\phi_1 1_{\mathbb T^N\setminus H_N}}
{\|\phi_1 1_{\mathbb T^N\setminus H_N}\|_{L^1(\mu_N)}}-\frac{\phi_2 1_{\mathbb T^N\setminus H_N}}
{\|\phi_2 1_{\mathbb T^N\setminus H_N}\|_{L^1(\mu_N)}}  \right|\right|_{L^2(\mu_N)}
&\leq
\frac{1}{1-(1+\delta_N)\sqrt{\mu_N(H_N)}}
\|\phi_1-\phi_2\|_{L^2(\mu_N)}
\\
&\quad+
\frac{2\sqrt{\mu_N(H_N)}}{(1-(1+\delta_N)\sqrt{\mu_N(H_N)})^2}
\|\phi_1-\phi_2\|_{L^2(\mu_N)}\\
&\le \frac{1- (\delta_N-1)\sqrt{\mu_N(H_N)}}{(1-(1+\delta_N)\sqrt{\mu_N(H_N)})^2}\|\phi_1-\phi_2\|_{L^2(\mu_N)}.
\end{align*}
Combining \eqref{eq:contraction-reduction} and the above, we obtain
\begin{equation}
\label{eq:open-contraction}
\|\mathcal {\bar P}_{\mu,H_N}(\phi_1)
-\mathcal {\bar P}_{\mu,H_N}(\phi_2)\|_{L^2(\mu_N)}
\leq 
\bar \gamma \left(\frac{1- (\delta_N-1)\sqrt{\mu_N(H_N)}}{(1-(1+\delta_N)\sqrt{\mu_N(H_N)})^2}\right)
\|\phi_1-\phi_2\|_{L^2(\mu_N)}.
\end{equation}
Because $\bar \gamma \in (0,1)$ and both \(\delta_N,\sqrt{\mu_N(H_N)}\to 0\),  as $N\to \infty$, \eqref{strongcontr} follows.
\end{proof}
In the following Proposition, we show a smoothing-type inequality, which will be key to establish the Sobolev estimate in \eqref{sobolev}.

\begin{proposicao}\label{prop:derivative-estimate-ANOVA}
Let $\{X_n\}_{n \ge 0}$ be a one-dimensional Markov chain satisfying Hypothesis \ref{B}, with its associated stationary measure $\mu$ and transfer operator $\mathcal{P}_{\mu}$.  Let $\mathcal{P}_{\mu,N}$ be the $N$-fold operator associated with $\mathcal{P}_{\mu}$ as in \eqref{uncoupledop}.
Let \(u \subset\{1,\dots,N\}\) be nonempty. Then, for every \(i\in u\),
\begin{align}
\|\partial_{i}\mathcal P_{\mu,N}(\phi)\|_{L^2(\mu_N)}
\le
 C_r  \bar \gamma ^{|u|-1}\,
\|\phi\|_{L^2(\mu_N)},\qquad \forall \phi \in S_u,
\label{eq:single-derivative-ANOVA}
\end{align}
with $C_r$ as in \eqref{regularization} and $\bar \gamma$ as in \eqref{one_step}. 
Furthermore, there exists a constant $\bar{C}_1$, independent of $N$, such that,  if $\phi \in L^2(\mu_N)$ has integral zero, then
\begin{align}
\sum_{i=1}^N
\|\partial_{i}\mathcal P_{\mu,N}(\phi)\|_{L^2(\mu_N)}^2
\le
\bar C_1 
\|\phi\|_{L^2(\mu_N)}^2.
\label{eq:global-gradient-bound}
\end{align}
\end{proposicao}

\begin{proof}
Let $u \subset \{1,\dots,N\}$ be nonempty. For each $j\in u$, let
$$
\mathcal P_{j}
:=
I^{\otimes(j-1)}
\otimes\mathcal P_{\mu}
\otimes I^{\otimes(N-j)}.
$$
For all $ i\in  u$, set $
\mathcal{P}_{u\setminus i} := \prod_{ j \in u\setminus \{i\}} \mathcal{P}_j.$
If $\phi \in S_u$, then $\partial_{i} \mathcal{P}_{\mu,N}(\phi)  = \partial_{i} \mathcal P_i (\mathcal P_{u\setminus i}(\phi)).$
A Fubini type argument analogous to the one used in \eqref{fubini_trick}, along with \eqref{regularization}, gives
\begin{equation}\label{applied_regul}
||\partial_{i} \mathcal{P}_{\mu,N}(\phi)||_{L^2(\mu_N)}^2 \le C_r^2 ||\mathcal P_{u\setminus i}(\phi)||_{L^2(\mu_N)}^2.
\end{equation}
By the definition of $S_u$ in \eqref{esseu}, if $\phi \in S_u$, then for $\mu$-a.e. $x_i \in \mathbb T$, the function 
$(x_k)_{k \in u\setminus i } \to \phi(x_i,(x_k)_{k \in u\setminus i }) \in S_{u \setminus i}$. Then, Lemma \ref{preserbationANOVA} applied with $\mathcal Q_{\mu}=\mathcal{P}_{\mu}$ and $\kappa =\bar \gamma$, combined with \eqref{applied_regul},  gives \eqref{eq:single-derivative-ANOVA}.

We now prove the smoothing estimate in \eqref{eq:global-gradient-bound}. Let $\phi \in L^2(\mu_N)$ be $\mu_N$-mean zero and let  
\[
\phi=\sum_{\varnothing\neq u\subset\{1,\dots,N\}} \phi_u
\]
be its ANOVA decomposition. Fix \(i\in\{1,\dots,N\}\). Then, by linearity
\begin{equation}\label{orthogonalterms}
\partial_{i}\mathcal P_{\mu,N}(\phi)
=
\sum_{u \colon i \in u }\partial_{i}\mathcal P_{\mu,N}(\phi_u)
\end{equation}
Note that the terms in the above sum are pairwise orthogonal.
Indeed, since $u\neq v$ and $i\in u\cap v$, after exchanging $u$ and $v$ if necessary we may choose $k\in u\setminus v$; necessarily $k\neq i$.. Then, if we denote $x_{-\{i,k\}} = (x_j)_{j \in \{1,\dots,N\}\setminus \{i,k\}}$, we have that
\begin{align*}
&\int \partial_{i}\mathcal P_{\mu,N}(\phi_u)(x)\partial_{i}\mathcal P_{\mu,N}(\phi_v)(x)d\mu_N(x) \\
&= \int \partial_{i}\mathcal P_{\mu,N}(\phi_v)(x_{-\{i,k\}},x_i) \left(\int \partial_{i}\mathcal P_{\mu,N}(\phi_u)(x_{-\{i,k\}},x_i,x_k) d\mu(x_k)\right) d\mu_{N-1}(x_{-\{i,k\}},x_i)\\
&=\int \partial_{i}\mathcal P_{\mu,N}(\phi_v)(x_{-\{i,k\}},x_i) \partial_{i}\left(\int \mathcal P_{\mu,N}(\phi_u)(x_{-\{i,k\}},x_i,x_k) d\mu(x_k)\right) d\mu_{N-1}(x_{-\{i,k\}},x_i) = 0,
\end{align*}
where in the last line we used that $\mathcal{P}_{\mu,N}(\phi_u) \in S_u \cap W^{1,2}(\mu_N)$, due to   \eqref{eq:single-derivative-ANOVA}.
Orthogonality of the terms in \eqref{orthogonalterms} implies that 
\begin{align}
\|\partial_{i}\mathcal P_{\mu,N}(\phi)\|_{L^2(\mu_N)}^2
=
\sum_{u \ni i}
\|\partial_{i}\mathcal P_{\mu,N}(\phi_u)\|_{L^2(\mu_N)}^2.
\label{eq:orthogonal-sum-fixed-i}
\end{align}
Using \eqref{eq:single-derivative-ANOVA} and \eqref{eq:orthogonal-sum-fixed-i},
\begin{align*}
\sum_{i=1}^N
\|\partial_{i}\mathcal P_{\mu,N}(\phi)\|_{L^2(\mu_N)}^2
&=
\sum_{i=1}^N
\sum_{u \ni i}
\|\partial_{i}\mathcal P_{\mu,N}(\phi_u)\|_{L^2(\mu_N)}^2\\
&\le
 C_r^2
\sum_{i=1}^N
\sum_{u\ni i}
\bar \gamma ^{2|u|-2}\,
\|\phi_u\|_{L^2(\mu_N)}^2\\
&=
 C_r^2
\sum_{\varnothing\neq u\subset\{1,\dots,N\}}
|u|\,\bar \gamma^{2|u|-2}\,
\|\phi_u\|_{L^2(\mu_N)}^2.
\end{align*}
Let
\begin{align*}
\bar C_1 := C_r^2 \sup_{m\ge 1} m\bar \gamma^{2m-2}<\infty.
\end{align*}
Then 
\begin{align*}
\sum_{i=1}^N
\|\partial_{i}\mathcal P_{\mu,N}(\phi)\|_{L^2(\mu_N)}^2
&\le
\bar{C}_1
\sum_{\varnothing\neq u\subset\{1,\dots,N\}}
\|\phi_u\|_{L^2(\mu_N)}^2\\
&=
\bar C_1
\|\phi\|_{L^2(\mu_N)}^2,
\end{align*}
where in the last step we used Parseval's identity for the orthogonal ANOVA decomposition.
This proves \eqref{eq:global-gradient-bound}.
\end{proof}
It remains to conclude the proof of Theorem \ref{thm2}
\begin{TheoremproofC}
Let $\{X_n\}_{n\ge 0}$ be a Markov chain that satisfies Hypothesis \ref{B} with stationary measure $\mu$, and let $\{H_N\}_{N\ge 0}$ be  a sequence of holes satisfying \eqref{holesdecay}. Let $B_N$ be as in \eqref{bienne}. By Proposition  \ref{inv_ne}, the conditioned operator $\mathcal{P}_{\mu,H_N}$ admits a non-negative density  $\psi_N \in B_N$ as an eigenfunction associated with the eigenvalue  
\begin{equation}\label{fittiest}
\lambda_N = 1 -\int_{H_N} \psi_N d\mu_N.
\end{equation}
Since $\psi_N \in B_N$, we have 
\begin{equation}\label{halfsobolev}
|\psi_N-1|_{L^2(\mu_N)} \le \delta_N, 
\end{equation}
with $\delta_N$ as in \eqref{deltaenne}.
Observe that, due to  \eqref{eq:global-gradient-bound}
\begin{align*}
||\nabla \psi_N||_{L^2(\mu_N)} &= ||\nabla(\psi_N-1)||_{L^2(\mu_N)} \\&= \left|\left|\nabla \left(\mathcal{P}_{\mu,N}\left(\frac{1_{\mathbb T^N\setminus H_N}  \psi_N}{||1_{\mathbb T^N\setminus H_N} \psi_N||_{L^1(\mu_N)}}-1\right)\right)\right|\right|_{L^2(\mu_N)}\\
&\le \sqrt{\bar{C}_1} \left|\left|\frac{1_{\mathbb T^N\setminus H_N}  \psi_N}{||1_{\mathbb T^N\setminus H_N} \psi_N||_{L^1(\mu_N)}}-1\right|\right|_{L^2(\mu_N)} \\
&\le \frac{\sqrt{\bar C_1}}{1-(1+\delta_N)\sqrt{\mu_N(H_N)}} (\left|\left|1_{\mathbb T^N\setminus H_N} \psi_N-1\right|\right|_{L^2(\mu_N)}+||1_{H_N}\psi_N||_{L^1(\mu_N)}) \\
& \le \frac{\sqrt{\bar C_1}}{1-(1+\delta_N)\sqrt{\mu_N(H_N)}} \left( ||\psi_N-1||_{L^2(\mu_N)}+ 2||1_{H_N}\psi_N||_{L^2(\mu_N)}\right) \\
&\le \frac{\sqrt{\bar C_1}}{1-(1+\delta_N)\sqrt{\mu_N(H_N)}} \left(  2||1_{H_N}||_{L^2(\mu_N)}+ 3|\psi_N-1|_{L^2(\mu_N)}\right) \\
&\le \frac{\sqrt{\bar C_1}}{1-(1+\delta_N)\sqrt{\mu_N(H_N)}} \left(2\sqrt{\mu_N(H_N)}+ 3\delta_N\right).
\end{align*}
Formula \eqref{sobolev} follows then from the definition of $\delta_N$ in \eqref{deltaenne} along with \eqref{halfsobolev}. Formula \eqref{rateofsurv} follows then from \eqref{sobolev} and \eqref{fittiest}.
\end{TheoremproofC}

\section{Proof of Corollary \ref{cor3}}\label{Section6}
\begin{TheoremproofB}

Let $\{Y_n\}_{n \ge0}$ be a Markov chain satisfying Hypothesis \ref{C}. By \cite[Theorem 1.1]{bhattacharya2002unique}, this chain has unique stationary measure $\mu$.
Since its transition kernel has density
$h(u-f(v))$ with respect to Lebesgue measure,  then, $\mu<<\text{Leb}$, and, if $\rho:= \frac{d\mu}{d\text{Leb}},$ invariance gives
\[
\rho(u)
 = \int_{\mathbb{T}} h(u-f(v))\,d\mu(v).
\]
By the standard regularity properties of convolutions
\cite[Proposition 8.10]{folland1999real} and the estimates for $h$ in  \eqref{estimates p},\eqref{diffest}, we have that  $\rho \in C^1(\mathbb{T})$ and satisfies 
\begin{equation}\label{doeb1}
c_1\le \rho(x) \le 1+c_2 \qquad \forall x \in \mathbb{T},
\end{equation}
along with
\begin{equation}\label{doeb2}
||\rho'||_{\infty} \le c_2.
\end{equation}
In particular, because of \eqref{doeb1}, $\mu$ is equivalent to Lebesgue measure.
The transfer operator $\mathcal{P}_{\mu}$ associated with $\{Y_n\}_{n\ge 0}$ reads
\begin{align*}
\mathcal{P}_{\mu}(\phi)(u)  
= \int \phi(v) \frac{h(u-f(v))}{\rho(u)}d\mu(v),
\end{align*}
where $h,f$ are as in Hypothesis \ref{C}. 
To prove Corollary \ref{cor3}, we need to show that $\mathcal P_{\mu}$ satisfies \eqref{one_step} and  \eqref{regularization}.
The regularization estimate in \eqref{regularization} follows from the identity 
\begin{align*}
\mathcal{P}_{\mu}(\phi)(u)' 
= \int \phi(v) \left(\frac{h'(u-f(v))}{\rho(u)}-\frac{\rho'(u) h(u-f(v))}{\rho(u)^2}\right)d\mu(v),
\end{align*}
and the estimates on $h$ and $\rho$ in \eqref{estimates p},\eqref{diffest},\eqref{doeb1},\eqref{doeb2}.

It remains to prove \eqref{one_step}. Let 
\begin{align*}
R(u,v):= \frac{h(u-f(v))}{\rho(u)}.
\end{align*}
Note that
\begin{equation}\label{intrdv}
\int R(u,v)d\mu(v) =   \int  \frac{h(u-f(v))}{\rho(u)}d\mu(v) = \mathcal{P}_{\mu}(1) = 1,
\end{equation}
and also
\begin{equation}\label{intrdu}
\int R(u,v) d\mu(u) = \int  \frac{h(u-f(v))}{\rho(u)}d\mu(u) = \int h(u-f(v))du  = 1.  
\end{equation}
Let $c:=\frac{c_1}{1+c_2} \in (0,1)$, with $c_1,c_2$ as in \eqref{estimates p} and \eqref{diffest}.
Let  $\bar R$  be defined by the formula
\begin{align*}
\frac{h(u-f(v))}{\rho(u)} = c + (1-c)\bar R(u,v).
\end{align*}
By the choice of $c$,  $\bar R \ge 0$ and, by \eqref{intrdu} and \eqref{intrdv}
\begin{equation}\label{erreprime}
\int \bar R(u,v) d\mu(u) = \int \bar R(u,v) d\mu(v) = 1.    
\end{equation}

If $\int \phi d\mu = 0$, then
\begin{align*}
\mathcal{P}_{\mu}(\phi)(u) &= c \int \phi(v) d\mu(v) + (1-c)\int \phi(v) \bar R(u,v) d\mu(v) \\
&=(1-c) \int \phi(v) \bar R(u,v) d\mu(v)  
\end{align*}
By \eqref{erreprime} and Jensen's inequality,
\begin{align*}
\int \mathcal{P}_{\mu}(\phi)(u)^2 d\mu(u) &\le (1-c)^2\int  \left(\int  \phi(v) \bar R(u,v) d\mu(v)  \right)^2   d\mu(u) \\
&\le (1-c)^2\int \int  \phi^2(v) \bar R(u,v) d\mu(v)   d\mu(u) \\
&\le  (1-c)^2\int \left(  \phi^2(v)  \int \bar R(u,v)  d\mu(u)\right) d\mu(v) \\
&\le (1-c)^2\int \phi^2(v) d\mu(v),
\end{align*}
which implies  
\begin{align*}
||\mathcal{P}_{\mu}(\phi)||_{L^2(\mu)}    \le (1-c)|\phi|_{L^2(\mu)}, \qquad \forall \phi \in L^2(\mu)\colon \int \phi d\mu =0,
\end{align*}
and in particular \eqref{one_step} holds with $\bar \gamma = 1-c$.
\end{TheoremproofB}

\appendix
\section{Keller-Liverani result}\label{appenix}
Let $\mathcal B\subset \widetilde{\mathcal B}$ be two Banach spaces satisfying 
\begin{equation}\label{continuousKL}
|f|_{\widetilde{\mathcal B}}
\le \|f\|_{\mathcal B},
\qquad \forall f\in\mathcal B.
\end{equation}

Let $\mathcal P_0$ be an operator 
admitting a rank-one projection $\Pi_0$, with $\Pi_0(1)=1$,  such that
\begin{equation}\label{mixingKL}
\|\mathcal P_0^n-\Pi_0\|_{\mathcal B\to\mathcal B}
\le C\gamma^n,
\qquad \forall n\ge0,
\end{equation}
for some  $C>0$ and $\gamma \in(0,1)$. Furthermore,  let $\mathcal P_\delta$ be an operator satisfying the Lasota--Yorke inequality
\begin{equation}\label{lasotayorkeKL}
\|\mathcal P_\delta^n f\|_{\mathcal B}
\le
C_1\gamma^n\|f\|_{\mathcal B}
+
C_2|f|_{\widetilde{\mathcal B}},
\qquad \forall n\ge0,
\end{equation}
for some $C_1,C_2>0$ and $\gamma \in (0,1)$ as in \eqref{mixingKL}.  Assume
\begin{equation}\label{nonespansionKl}
\max\{||\mathcal{P}_0||_{\mathcal{B}\to\mathcal{B}},|\mathcal{P}_0|_{\widetilde{\mathcal{B}}\to \widetilde{\mathcal{B}}},||\Pi_0||_{\mathcal{B}\to\mathcal{B}},|\Pi_0|_{\widetilde{\mathcal{B}}\to \widetilde{\mathcal{B}}},||\mathcal{P}_{\delta}||_{\mathcal{B}\to\mathcal{B}},|\mathcal{P}_{\delta}|_{\widetilde{\mathcal{B}}\to \widetilde{\mathcal{B}}}\} \le 1
\end{equation}
This assumption is unnecessary, but simplifies significantly the computations, and is satisfied in our setting.

The Keller-Liverani result in \cite{KellerLiverani1999} implies that, for all $\gamma' \in (\gamma,1)$, there exists a $\delta_0$, depending on $\gamma'$ and the constants in \eqref{continuousKL}-\eqref{lasotayorkeKL}, such that, if  $ \delta_0>\delta>0$ satisfies
\begin{equation}\label{perturbationKL}
|(\mathcal P_\delta-\mathcal P_0)f|_{\widetilde{\mathcal B}}
\le
\delta\|f\|_{\mathcal B},
\qquad \forall f\in\mathcal B,
\end{equation}
then, there exist  functions $F=F(\delta),G=G(\delta)$ such that 
\begin{align*}
\sigma(\mathcal P_\delta)
\subset
B_{\gamma'}(0)
\cup
B_{F(\delta)}(1).
\end{align*}
and 
\begin{align*}
|\psi_\delta-1|_{\widetilde{\mathcal B}}
\le
G(\delta),
\end{align*}
where $\psi_{\delta}$ denotes the unique eigenfunction associated with the unique eigenvalue $\lambda_{\delta}\in B_{F(\delta)}(1)$, normalized so that $\Pi_0(\psi_{\delta}) =1$.

The following proposition  follows the computations in \cite{KellerLiverani1999}, in order to find explicit estimates $\delta_0,F,G$, depending on the constants in \eqref{continuousKL}-\eqref{lasotayorkeKL}.

\begin{proposicao}\label{Kl}
Let  $\mathcal{B},\widetilde{\mathcal B},\mathcal{P}_0,\mathcal{P}_{\delta}$ satisfy \eqref{continuousKL}-\eqref{nonespansionKl}, $\gamma'\in(\gamma,1)$, with $\gamma$ as in \eqref{mixingKL} 
and $\beta>0$ satisfying
$0<\beta<\frac{\eta}{2}$, with
\begin{equation}\label{etaKL}
\eta
:=
\frac{\log(\gamma'/\gamma)}
{\log(1/\gamma)}.
\end{equation}
Let $\delta$ as in \eqref{perturbationKL}. Then, there exist $\delta_0>0$, $\bar C_{\beta,\gamma'}>0$, depending only on the constants in \eqref{continuousKL}-\eqref{lasotayorkeKL}, such that, if $\delta<\delta_0$, then 
\begin{equation}\label{spectrumKL}
\sigma(\mathcal P_{\delta})
\subset
B_{\gamma'}(0)
\cup
B_{\delta^\beta}(1),
\end{equation}
and  $B_{\delta^\beta}(1)$ contains a unique algebraically simple eigenvalue $\lambda_\delta$ of $\mathcal P_\delta$.
Furthermore, the eigenfunction $\psi_\delta$ associated with $\lambda_\delta$, normalized so that $
\Pi_0(\psi_\delta)=1$
satisfies
\begin{equation}\label{eigenfunctionKL}
|\psi_\delta-1|_{\widetilde{\mathcal B}}
\le
\bar C_{\beta,\gamma'}\delta^{\eta-\beta}.
\end{equation}
\end{proposicao}

\begin{proof}
The proof follows the argument in \cite[Theorem 1]{KellerLiverani1999}, keeping track of the dependence of the estimates on $\delta$. For $z\in\mathbb C$, define $A_0(z):=zI-\mathcal P_0$ and $A_\delta(z):=zI-\mathcal P_\delta.$
For $z\in\mathbb{C}$ such that $A_0(z)$ and $A_\delta(z)$ are invertible, denote their inverses by $R_0(z)$ and $R_\delta(z)$, respectively.

\emph{Step 1: a strong a priori estimate.}

Let $g=A_\delta(z)f.$ Then, for every $m\ge1$ and $z \in \mathbb{C}$
\begin{align*}
z^mf
=
\mathcal P_\delta^m f
+
\sum_{j=0}^{m-1}
z^{m-1-j}\mathcal P_\delta^jg.
\end{align*}
 Using \eqref{lasotayorkeKL} and \eqref{nonespansionKl}, if $|z|\ge\gamma'$, we obtain
\begin{align*}
|z|^m\|f\|_{\mathcal B}
&\le
C_1\gamma^m\|f\|_{\mathcal B}
+
C_2|f|_{\widetilde{\mathcal B}}
+
\sum_{j=0}^{m-1}
|z|^{m-1-j}\|\mathcal P_\delta^jg\|_{\mathcal B}.
\end{align*}
Choosing, in the above, $m_0\ge1$ such that $C_1\left(\frac{\gamma}{\gamma'}\right)^{m_0}
\le
\frac12,$ and using again \eqref{lasotayorkeKL}, \eqref{continuousKL} and \eqref{nonespansionKl} we find $D_1,D_2>0$, independent of $\delta$ and $z$, such that
\begin{equation}\label{aprioriKL}
\|f\|_{\mathcal B}
\le
D_1\|A_\delta(z)f\|_{\mathcal B}
+
D_2|f|_{\widetilde{\mathcal B}}.
\end{equation}

\emph{Step 2: spectral gap of the perturbed operator.}

Let $z \in\mathbb C$ such that $|z| \in (\gamma',2)$  and $|1-z|> \delta^\beta$. Note that, for  every $n\ge1$,
\begin{align*}
A_0(z)^{-1}
=
z^{-n}A_0(z)^{-1}\mathcal P_0^n
+
\sum_{j=0}^{n-1}z^{-j-1}\mathcal P_0^j.
\end{align*}
Let $g=A_0(z)h$. By the estimates for the unperturbed operator in  \eqref{mixingKL}, one can prove that there exists $C_0>0$ such that for our choice of $z$,
\begin{equation}\label{resolvent0KL}
\|R_0(z)\|_{\mathcal B\to\mathcal B}
\le C_0\left(
\frac{1}{\delta^{\beta}}
+
\frac{1}{\gamma'-\gamma}\right),
\end{equation}
Using \eqref{mixingKL} and the above, we obtain that there exists $C$ that depends on $\gamma'$, such that
\begin{align}
|h|_{\widetilde{\mathcal B}}
&\le C\delta^{-\beta}\left(
\left(\frac{\gamma}{\gamma'}\right)^n
\|g\|_{\mathcal B}
+
(\gamma')^{-n}|g|_{\widetilde{\mathcal B}}\right).
\label{weakunperturbedKL}
\end{align}

Furthermore, observe that  $g =  (\mathcal P_\delta-\mathcal P_0)h + A_\delta(z)h.$  Hence, by \eqref{perturbationKL},
\begin{equation}\label{gcomparisonKL}
|g|_{\widetilde{\mathcal B}}
\le
\delta \|h\|_{\mathcal B}
+
|A_\delta(z)h|_{\widetilde{\mathcal B}}.
\end{equation}

Set $$
m_\delta
:=
\left\lceil
\frac{\log(\delta^{-1})}
{\log(\gamma^{-1})}
\right\rceil.$$ 
Then, there exists $\bar C>0,$ independent of $\delta$, such that 
\[
\gamma^{m_\delta}
\le
\bar C\delta,
\qquad
\left(\frac{\gamma}{\gamma'}\right)^{m_\delta}
\le
\bar C\delta^\eta,
\qquad
\delta(\gamma')^{-m_\delta}
\le
\bar C\delta^\eta.
\]

Taking $n=m_\delta$ in \eqref{weakunperturbedKL}, using the above estimates, \eqref{gcomparisonKL} and then \eqref{aprioriKL}, we obtain constants $C_1,C_2,C_3>0$, independent of $\delta$, such that
\begin{align*}
|h|_{\widetilde{\mathcal B}}
\le
C_1\delta^{\eta-\beta}\|h\|_{\mathcal B}
+
C_2\delta ^{\eta-\beta}|h|_{\widetilde{\mathcal B}}
+
C_3\delta^{\eta-1-\beta}
|A_\delta(z)h|_{\widetilde{\mathcal B}}.
\end{align*}
Suppose that $\delta$ satisfies 
\begin{equation}\label{first_delta_bound}
C_2\delta^{\eta-\beta}<\frac{1}{2}.
\end{equation}
Then, for such $\delta$
\begin{equation}\label{weakresolventKL}
|h|_{\widetilde{\mathcal B}}
\le
2C_1\delta^{\eta-\beta}\|h\|_{\mathcal B}
+
2C_3\delta^{\eta-1-\beta}
|A_\delta(z)h|_{\widetilde{\mathcal B}}.
\end{equation}
Combining this estimate with \eqref{aprioriKL} and arguing as in \cite[Remark~6]{KellerLiverani1999}, we obtain, after reducing
$\delta_0$ if necessary, that $R_\delta(z)$ exists and is a bounded linear operator on $\mathcal{B}$ for all $0<\delta<\delta_0$,
$|z|\in(\gamma',2)$ and $z\notin B_{\delta^\beta}(1)$. 
As a result, we have
\begin{equation}\label{spectralcontainmentKL}
\sigma(\mathcal P_\delta)
\subset
B_{\gamma'}(0)
\cup
B_{\delta^\beta}(1).
\end{equation}

\emph{Step 3: comparison of the resolvents.}

Let $g:= (\mathcal P_\delta-\mathcal P_0)R_0(z)f.$  Applying \eqref{weakresolventKL} to $R_\delta(z)g$, we obtain 
\begin{align*}
|R_{\delta} g|_{\widetilde{\mathcal B}} \le 2C_1 \delta^{\eta-\beta}  ||R_{\delta}g||_{\mathcal B} + 2C_3\delta^{\eta-1-\beta}|g|_{\widetilde{\mathcal{B}}}
\end{align*}
Combining the above with \eqref{aprioriKL}, we get 
\begin{align*}
|R_{\delta}g|_{\widetilde{\mathcal B}} \le 2C_1 \delta^{\eta-\beta} D_1 ||g||_{\mathcal B}  + 2C_1 \delta^{\eta-\beta} D_2 |R_{\delta} g|_{\widetilde{\mathcal B}} + 2C_3\delta^{\eta-1-\beta}|g|_{\widetilde{\mathcal{B}}}.
\end{align*}
If $\delta$ is small enough so that
\begin{equation}\label{second_step_delta}
2C_1 \delta^{\eta-\beta} D_2 <\frac12,
\end{equation}
using the resolvent equation $R_\delta(z)-R_0(z)
=
R_\delta(z)
(\mathcal P_\delta-\mathcal P_0)
R_0(z),$ and the above estimate, we get 
\begin{align*}
|(R_\delta(z)-R_0(z))f|_{\widetilde{\mathcal B}}
&\le 4C_1 \delta^{\eta-\beta} D_1 ||g||_{\mathcal B}  + 4C_3\delta^{\eta-1-\beta}|g|_{\widetilde{\mathcal{B}}} \\
&\le (4C_1 C  D_1 +4C_3C )\delta^{\eta-2\beta}
\|f\|_{\mathcal B},
\end{align*}
where in the last inequality we used \eqref{perturbationKL} and \eqref{resolvent0KL}.
Set  $\bar D := 4C_1 C  D_1 +4C_3C.$ Then, we conclude 
\begin{equation}\label{resolventcomparisonKL}
|(R_\delta(z)-R_0(z))f|_{\widetilde{\mathcal B}}
\le \bar D\delta^{\eta-2\beta}
\|f\|_{\mathcal B}.
\end{equation}

\emph{Step 4: the spectral projection near $1$.}
Let $\delta$ be small enough such that
\begin{equation}
\delta^\beta
<
\frac{1-\gamma'}{2},
\end{equation}
and let  $\Gamma_\delta$ be the boundary of the disk of radius $2 \delta^{\beta}$ around $1$. Define the Riesz projections 
\begin{align*}
\Pi_\delta
&:=
\frac{1}{2\pi i}
\int_{\Gamma_\delta}
R_\delta(z)\,dz,
\\
\Pi_0
&:=
\frac{1}{2\pi i}
\int_{\Gamma_\delta}
R_0(z)\,dz.
\end{align*}
By \eqref{resolventcomparisonKL},
\begin{equation}\label{projectionestimateKL}
\begin{aligned}
|\Pi_\delta-\Pi_0|_{\mathcal B\to\widetilde{\mathcal B}}
&\le
\frac{1}{2\pi}
\int_{\Gamma_\delta}
|R_\delta(z)-R_0(z)|_{\mathcal B\to\widetilde{\mathcal B}}
\,|dz|
\\
&\le
C\delta^\beta
\delta^{\eta-2\beta}
\\
&=
C\delta^{\eta-\beta}.
\end{aligned}
\end{equation}

\emph{Step 5: estimate for the eigenfunction.}
By \cite[Remark 4]{KellerLiverani1999}, if  $\delta$ is small enough, the rank of the spectral projection is preserved. Since $\Pi_0$ has rank one, $B_{\delta^\beta}(1)$ contains exactly one spectral value $\lambda_\delta$, counted with algebraic multiplicity. In particular, $\lambda_\delta$ is algebraically simple. Set $\widehat{\psi}_\delta:=\Pi_\delta(1).$ Since $\Pi_0(1)=1$, by \eqref{projectionestimateKL},
\begin{align*}
|\widehat{\psi}_\delta-1|_{\widetilde{\mathcal B}}
\le
C\delta^{\eta-\beta}.
\end{align*}
Since $\Pi_0$ has rank one, let $a_\delta\in\mathbb C$ be such that $\Pi_0(\widehat{\psi}_\delta)=a_\delta 1.$

Then
\begin{align*}
|a_\delta-1|
\le
C|\widehat{\psi}_\delta-1|_{\widetilde{\mathcal B}}
\le
C\delta^{\eta-\beta}.
\end{align*}
Hence, the normalized eigenfunction  $\psi_\delta
:=
a_\delta^{-1}\widehat{\psi}_\delta$ satisfies
\begin{align*}
|\psi_\delta-1|_{\widetilde{\mathcal B}}
&\le
|a_\delta^{-1}|
|\widehat{\psi}_\delta-1|_{\widetilde{\mathcal B}}
+
|a_\delta^{-1}-1|
|1|_{\widetilde{\mathcal B}}
\\
&\le
C_{\beta,\gamma'}
\delta^{\eta-\beta}.
\end{align*}
This concludes the proof.
\end{proof}
\bibliographystyle{plain} 
\bibliography{biblio}
\end{document}